\documentclass[12pt, a4paper]{amsart}
\usepackage{amscd,amssymb,mathtools}

\usepackage{float}

\usepackage{enumitem}
\usepackage{tensor}
\usepackage{wrapfig}
\usepackage{tikz}

\allowdisplaybreaks

\usetikzlibrary{cd}

\usepackage{color, url, hyperref}

\hypersetup{colorlinks=true, linkcolor=blue, linkbordercolor=blue, citecolor=blue, citebordercolor=blue, linktocpage=true}

\def\Aut{\operatorname{Aut}}

\def\ker{\operatorname{ker}}

\def\Pp{\mathbb{P}}

\def\id{\operatorname{id}}

\def\QAut{\operatorname{QAut}}
\def\Sym{\operatorname{Sym}}
\def\SU{\operatorname{SU}}

\def\A{\mathbb{A}}

\def\C{\mathbb{C}}

\def\N{\mathbb{N}}

\def\AA{\mathcal{A}}

\newtheorem{thm}{Theorem}[section]

\newtheorem{lemma}[thm]{Lemma}
\newtheorem{prop}[thm]{Proposition}

\theoremstyle{definition}
\newtheorem{definition}[thm]{Definition}

\theoremstyle{remark}
\newtheorem{remark}[thm]{Remark}

\newtheorem{remarks}[thm]{Remarks}
\newtheorem{example}[thm]{Example}

\numberwithin{equation}{section}

\tikzstyle{vertex}=[circle]
\tikzstyle{goto}=[->,shorten >=1pt,>=stealth,semithick]

\begin{document}

\date{\today}
\title[Self-similar quantum groups]{Self-similar quantum groups}
\author{Nathan Brownlowe}
\address{Nathan Brownlowe, School of Mathematics and Statistics, University of Sydney,
NSW 2006, Australia.}
\email{nathan.brownlowe@sydney.edu.au}

\author{David Robertson}
\address{David Robertson, School of Science and Technology, University of New England, NSW 2351, Australia.}
\email{david.robertson@une.edu.au}

 \thanks{Brownlowe was supported by the Australian Research Council grant DP200100155, and both authors were supported by the Sydney Mathematical Research Institute.}

\subjclass[2010]{}
\thanks{}

\begin{abstract}
We introduce the notion of self-similarity for compact quantum groups. For a finite set $X$, we introduce a $C^*$-algebra $\A_X$, which is the quantum automorphism group of  the infinite homogeneous rooted tree $X^*$. Self-similar quantum groups are then certain quantum subgroups of $\A_X$. Our main class of examples are called finitely-constrained self-similar quantum groups, and we find a class of these examples that can be described as quantum wreath products by subgroups of the quantum permutation group.
\end{abstract}

\maketitle

\setcounter{tocdepth}{1}


\section{Introduction}\label{sec:intro}

Self-similar groups are a class of groups acting faithfully on an infinite rooted homogeneous tree $X^*$. In particular, given an automorphism $g\in\Aut(X^*)$ and a vertex $w\in X^*$, by identifying $wX^*$ with $g(w)X^*$, we get an automorphism $g|_w\in\Aut(X^*)$ which is uniquely determined by the identity
\[
g\cdot (wv)=(g\cdot w)g|_w\cdot v\quad\text{for all }v\in X^*.
\]
The automorphism $g|_w$ is called the \textit{restriction} of $g$ by $w$, and a subgroup $G\le \Aut(X^*)$ is \textit{self-similar} if it is closed under restrictions. Self-similar groups are a significant class of groups that play an important role in geometric group theory, and have been a rich source of groups displaying interesting phenomena. Most notably, the Grigorchuk group \cite{Gr80} is a self-similar group which is  an infinite, finitely generated periodic group and provided the first example of a group with intermediate growth, as well as the first known amenable group to not be elementary amenable.

When the group of automorphisms $\Aut(X^*)$ is equipped with the permutation topology, the closed self-similar groups are examples of compact, totally disconnected groups, and hence are profinite groups.  A particular class of examples of interest are the self-similar groups of \textit{finite type}, which are subgroups of automorphisms of $X^*$ that act like elements of a given finite group locally around every vertex. Grigorchuk introduced this concept in \cite{Gr05}, where he also showed that the closure of the Grigorchuk group is a self-similar group of finite type. Note that these groups are called finitely constrained self-similar groups in \cite{Sunic11}, and we will use that terminology. 

The theory of compact quantum groups is by now a very substantial part of the wider field of quantum groups, and one which sits in the framework of operator algebras. The theory started with Woronowicz's introduction of the quantum $\SU(2)$ group in \cite{WorSU2}. Woronowicz then defined compact matrix quantum groups in \cite{Wor87}, before developing a general theory of compact quantum groups in \cite{Wor95}. An important class of compact matrix quantum groups was identified and studied by Wang through his quantum permutation groups in \cite{Wang1998}. Wang was motivated by one of Connes' questions from his noncommutative geometry program: what is the \textit{quantum} automorphism group of a space? Wang's work in \cite{Wang1998} provided an answer for finite spaces; in particular, Wang formally defined the notion of a quantum automorphism group, and then showed that his quantum permutation group $A_s(n)$ is the quantum automorphism group of the space with $n$ points. For three or fewer points this algebra is commutative, and hence indicating no quantum permutations; but for four or more points, remarkably the algebra is noncommutative and infinite-dimensional. 

Since the appearance of \cite{Wang1998}, follow-up work progressed in multiple directions, including the results of Bichon in \cite{Bichon03} in which he introduced  quantum automorphisms of finite graphs. These algebras are quantum subgroups of the quantum permutation groups. Bichon used this construction to define the quantum dihedral group $D_4$. Later still in \cite{BB09}, Banica and Bichon classified all the compact quantum groups acting on four points; that is, all the compact quantum subgroups of $A_s(4)$. Quantum automorphisms of infinite graphs have recently been considered by Rollier and Vaes in \cite{RV}, and by Voigt in \cite{V}.

Our current work is the result of us asking the question: is there a reasonable notion of self-similarity for \textit{quantum} groups? We answer this question in the affirmative for compact quantum groups. We do this by first constructing the quantum automorphism group $\A_X$ of the homogeneous rooted tree $X^*$, and then identifying the quantum analogue of the restriction maps $g\mapsto g|_w$ for $g\in\Aut(X^*)$, $w\in X^*$. We then define a self-similar quantum group to be any quantum subgroup $A$ of $\A_X$ for which the restriction maps factor through the quotient map $\A_X\to A$. We characterise self-similar quantum groups in terms of a certain homomorphism $A\otimes C(X)\to C(X)\otimes A$, which can be thought of as quantum state-transition function. The main class of examples we examine are quantum analogues of finitely constrained self-similar groups. In our main theorem about these examples we describe a class of finitely constrained self-similar groups as free wreath products by quantum subgroups of quantum permutation groups. 

We start with a small preliminaries section in which we collect all the required definitions from the literature on compact quantum groups. In Section \ref{sec:AX} we then identify a compact quantum group $\A_X$ which we prove is the quantum automorphism group of the homogeneous rooted tree $X^*$. The $C^*$-algebra $\A_X$ is a noncommutative, infinite-dimensional $C^*$-algebra whose abelianisation is the algebra of continuous functions on the automorphism group of the tree $X^*$. In Section~\ref{sec:selfsimilarity} we introduce the notion of self-similarity for compact quantum groups, and we characterise self-similar quantum groups $A$ in terms of morphisms $A\otimes C(X)\to C(X)\otimes A$, mimicking the fact that classical self-similar action are governed by the maps $G\times X\to X\times G\colon (g,x)\mapsto (g\cdot x,g|_x)$. 

In Section~\ref{sec:finitelyconstrained} we define finitely constrained self-similar quantum groups, which are the quantum analogues of the classical finitely constrained self-similar groups studied in \cite{BS13,Sunic11}. In particular, we consider subalgebras $\A_d$ of $\A_X$, which are the quantum automorphism groups of the finite subtrees $X^{[d]}$ of $X^*$ of depth $d$. To each quantum subgroup $\Pp$ of $\A_d$, we construct a quantum subgroup $A_\Pp$, which we prove is a self-similar quantum group. We then build on the work of Bichon in \cite{Bichon04} by constructing free wreath products of compact quantum groups by quantum subgroups of the quantum permutation group (which corresponds to the subalgebra $\A_1$ of $\A_X$), and we prove that every $A_\Pp$ coming from a quantum subgroup $\Pp$ of $\A_1$ is canonically isomorphic to the free wreath product $A_\Pp\ast_w\Pp$. 

	\section{Preliminaries}\label{sec:prelim}
	
	In this section we collect some basics on compact quantum groups. We start with Woronowicz's definition of a compact quantum group \cite{Wor95}.
	
	\begin{definition} \label{CQG:densitydefn}
	A \emph{compact quantum group} is a pair $(A,\Phi)$ where $A$ is a unital $C^*$-algebra and $\Phi : A \to A \otimes A$ is a unital $*$-homomorphism such that
\begin{enumerate}
 \item \label{cqg:coassoc} $(\Phi\otimes \id)\Phi = (\id\otimes\Phi) \Phi$
 \item $\overline{(A\otimes 1)\Phi(A)} = A \otimes A = \overline{(1\otimes A)\Phi(A)}$.
\end{enumerate}
We call $\Phi$ the \emph{comultiplication} and (\ref{cqg:coassoc}) is called \emph{coassociativity}.
\end{definition}

\begin{remark} \label{CQG:matrixdefn}
It is proved in \cite{Wor95} that $(A,\Phi)$ is a compact quantum group if and only if there is a family of matrices $\{a^\lambda = (a^\lambda_{i,j}) \in M_{d_\lambda}(A) : \lambda \in \Lambda\}$ for some indexing set $\Lambda$ such that
\begin{enumerate}
\item \label{CQG:matrixdef-comult} $\Phi(a^\lambda_{i,j}) = \displaystyle\sum_{k=1}^{d_\lambda} a^\lambda_{i,k} \otimes a^{\lambda}_{k,j}$ for all $\lambda\in\Lambda$ and $1\leq i,j \leq d_\lambda$,
\item \label{CQG:matrixdef-inverse} $a^\lambda$ and its transpose  $(a^\lambda)^T$ are invertible elements of $M_{d_\lambda}(A)$ for every $\lambda\in\Lambda$,
\item \label{CQG:matrixdef-dense}the $*$-subalgebra $\mathcal{A}$ of $A$ generated by the entries $\{a^\lambda_{i,j} : 1\leq i,j \leq d_\lambda, \lambda\in \Lambda\}$ is dense in $A$. \end{enumerate}
\end{remark}

\begin{example}\label{eg:permutation}
	A key example for us is Wang's \textit{quantum permutation groups} $(A_s(n),\Phi)$ from \cite{Wang1998}. Here, $n$ is a positive integer, and $A_s(n)$ is the universal $C^*$-algebra generated by elements $a_{ij}$, $1\le i,j\le n$, satisfying \begin{align*}
		& a_{ij}^2=a_{ij}=a_{ij}^*\text{ for all $1\le i,j\le n$,}\\
		& \sum_{j=1}^na_{ij}=1 \text{ for all $1\le i\le n$,}\\
		& \sum_{i=1}^na_{ij}=1 \text{ for all $1\le j\le n$.}
	\end{align*}
The comultiplication $\Phi$ satisfies $\Phi(a_{ij})=\sum_{k=1}^na_{ik}\otimes a_{kj}$ for all $1\le i,j\le n$. 
\end{example}

\begin{definition}
	If $(A_1,\Phi_1)$ and $(A_2,\Phi_2)$ are compact quantum groups, then a \textit{morphism} $\pi$ from $(A_1,\Phi_1)$ to $(A_2,\Phi_2)$ is a homomorphism of $C^*$-algebras $\pi\colon A_1\to A_2$ satisfying $(\pi\otimes \pi)\circ \Phi_1=\Phi_2\circ\pi$.
\end{definition}

\begin{definition}
Let $(A,\Phi)$ be a compact quantum group. A \textit{Woronowicz ideal} is an ideal $I$ of $A$ such that $\Phi(I)\subseteq \ker(q\otimes q)$, where $q$ is the quotient map $A\to A/I$. Then $(A/I,\Phi')$, where $\Phi'\colon A/I\to A/I\otimes A/I$ satisfies $\Phi'\circ q=(q\otimes q)\circ \Phi$ is a compact quantum group called a \textit{quantum subgroup} of $(A,\Phi)$.
\end{definition}

\begin{definition}
A \emph{(left) coaction} of a compact quantum group $(A,\Phi)$ on a unital $C^*$-algebra $B$ is a unital $*$-homomorphism $\alpha : B \to A \otimes B$ satisfying
\begin{enumerate}
 \item \label{coactionid} $(\id\otimes\alpha)\alpha = (\Phi\otimes\id)\alpha$
 \item \label{coactiondensity} $\overline{\alpha(B)(A\otimes 1)} = A \otimes B$.
\end{enumerate}
We refer to (\ref{coactionid}) as the \emph{coaction identity} and (\ref{coactiondensity}) is known as the \emph{Podle\' s condition}.
\end{definition}
	
	\section{Quantum automorphisms of a homogeneous rooted tree}\label{sec:AX}
	
	In this section we introduce a compact quantum group $\A_X$ which we prove is the quantum automorphism group of the infinite homogeneous rooted tree $X^*$. We start with the notion of an action of a compact quantum group on $X^*$. Note that for $n\ge 0$ we write $X^n$ for all the words in $X$ of length $n$, and we then the tree $X^*$ can be identified with $\bigcup_{n\ge }X^n$, where $X^0=\{\varnothing\}$ and $\varnothing$ is the root of the tree.
	
	\begin{definition}
	Let $X$ be a finite set and let $(A,\Phi)$ be a compact quantum group. An action of $A$ on the homogeneous rooted tree $X^*$ is a system
	\[
	 \alpha = (\alpha_n : C(X^n) \to A \otimes C(X^n))
	\]
of left coactions, such that for any $m < n$ the diagram
		\begin{center}
			\begin{tikzcd}
				C(X^m) \arrow[r, "i_{m,n}"] \arrow[d, "\alpha_m"]
				& C(X^n) \arrow[d, "\alpha_n"] \\
				A \otimes C(X^m) \arrow[r, "\id \otimes i_{m,n}"]
				& A \otimes C(X^n)
			\end{tikzcd}
		\end{center}
commutes, where $i_{m,n} : C(X^m) \to C(X^n)$ is the injective homomorphism satisfying
\[
 i_{m,n}(p_w) = \sum_{w' \in X^{n-m}} p_{ww'}.
\]
	\end{definition}

We now define the main object of interest in this section, the $C^*$-algebra $\A_X$, before proving that it is indeed a compact quantum group in Theorem~\ref{thm:AXisaCQG}. At some point in the later stages of this project we became aware of \cite{RV}, and their notion of the quantum automorphism group $\QAut\Pi$ of a locally finite connected graph $\Pi$. A straightforward argument shows that $\A_X$ is $\QAut\Pi$ for $\Pi$ the homogeneous rooted tree, but we include the proof of Theorem~\ref{thm:AXisaCQG} for completeness. 
	
	\begin{definition}\label{def:qaut}
		Let $X$ be a finite set. Define $\mathbb{A}_X$ to be the universal $C^*$-algebra generated by elements $\{a_{u,v}: u,v\in X^n, n\geq 0\}$ subject to the following relations:
		\begin{enumerate}
			\item \label{qaut:identity} $a_{\varnothing,\varnothing} = 1$,
			\item \label{qaut:proj} for any $n\geq 0, u,v\in X^n$, $a_{u,v}^* = a_{u,v}^2 = a_{u,v}$,
			\item \label{qaut:sum} for any $n\geq 0, u,v\in X^n$ and $x\in X$
			\[
			a_{u,v} = \sum_{y\in X} a_{ux,vy} = \sum_{z\in X} a_{uz,vx}.
			\]
		\end{enumerate}
	\end{definition}

\begin{remarks}\label{rmks:aboutAX}
	\begin{itemize}
		\item[(i)] For each $d\in\mathbb{N}$ we denote by $\A_d$ the subalgebra of $\A_X$ generated by $\{a_{u,v}:u,v\in X^d\}$. Note that $\A_1$ is the Wang's quantum permutation group $A_s(|X|)$ from Example~\ref{eg:permutation}.
		
		\item[(ii)] We can interpret (\ref{qaut:sum}) as follows: each projection $a_{u,v}$ decomposes as an $|X|\times |X|$ square of projections $\{a_{ux,vy}:x,y\in X\}$ with a magic square type property where every row and column sums to $a_{u,v}$. For example, if $X=\{0,1,2\}$ we have the following structure.
		\begin{equation*}
			a_{u,v} \mapsto
			\scalebox{0.6}{
				\begin{tikzpicture}[baseline={([yshift=-.5ex]current bounding box.center)},vertex/.style={anchor=base,circle,fill=black!25,minimum size=18pt,inner sep=2pt},declare function={boxW=1.3cm;},
					box/.style={minimum size=boxW,draw}]  
					
					\foreach \x in {0,...,2}
					{
						\foreach \y in {0,...,2}
						{
							\node [box] at ({\x*boxW}, {-\y*boxW}) {$a_{u\y,v\x}$};
						}
					}
				\end{tikzpicture}}
		\end{equation*}
	
		\item[(iii)] Repeated applications of (3) from Definition~\ref{def:qaut} show that for all $u,u',v,v',w\in X^n, n\in\mathbb{N}$, we have		
		\[
		u\neq u',v\neq v' \implies a_{u,w}a_{u',w} = 0 = a_{w,v}a_{w,v'},
		\]
		and that for all $u=u_1\cdots u_n,v=v_1\cdots v_n\in X^n, n\in\N$, and $x,y\in X$ we have
		\[
		a_{x,y}a_{u,v}=a_{u,v}a_{x,y}=\begin{cases}a_{u,v} &\text{if $u_1=x,v_1=y$}\\0 &\text{otherwise.}\end{cases}
		\]
		We will freely use these two identities without comment throughout the rest of the paper.
	\end{itemize}
\end{remarks}

%
	
	\begin{thm}\label{thm:AXisaCQG}
		The $C^*$-algebra $\mathbb{A}_X$ is a compact quantum group with comultiplication $\Delta : \mathbb{A}_X \to \mathbb{A}_X\otimes \mathbb{A}_X$ satisfying
		\[
		\Delta(a_{u,v}) = \sum_{w\in X^n} a_{u,w} \otimes a_{w,v},
		\]
		for all $u,v\in X^n$ and $n\geq 1$
	\end{thm}
	
	\begin{proof}
		To see that $\Delta$ exists, it's enough to show that the elements
		\[
		b_{u,v} := \sum_{w\in X^n} a_{u,w} \otimes a_{w,v}
		\]
		for $u,v\in X^n$ and $n\geq 1$ satisfy Definition \ref{def:qaut}.
		
		Firstly, $b_{\varnothing,\varnothing} = \Delta(a_{\varnothing,\varnothing}) = a_{\varnothing,\varnothing}\otimes a_{\varnothing,\varnothing} = 1\otimes 1$. For (\ref{qaut:proj}), we have
		\[
			b_{u,v}^* = \sum_{w\in X^n} a_{u,w}^* \otimes a_{w,v}^* 
			= \sum_{w\in X^n} a_{u,w} \otimes a_{w,v} 
			= b_{u,v} 
		\]
		and
		\begin{align*}
			b_{u,v}^2 &= \left(\sum_{w\in X^n} a_{u,w} \otimes a_{w,v}\right)^2 \\
			&= \sum_{w,z\in X^n} a_{u,w} a_{u,z} \otimes a_{w,v} a_{z,v} \\
			&= \sum_{w\in X^n} a_{u,w}^2 \otimes a_{w,v}^2 \\
			&= \sum_{w\in X^n} a_{u,w} \otimes a_{w,v}\\
			&= b_{u,v}.
		\end{align*}
		For (\ref{qaut:sum}) fix $u,v \in X^n$ and $x\in X$. Then
		\begin{align*}
			b_{u,v} &= \sum_{w\in X^n} a_{u,w} \otimes a_{w,v} \\
			&= \sum_{w\in X^n} \sum_{z\in X} a_{ux,wz} \otimes a_{w,v} \\
			&= \sum_{w\in X^n} \sum_{z\in X} a_{ux,wz} \otimes \sum_{y\in X} a_{wz,vy} \\
			&= \sum_{y\in X} \sum_{w\in X^{n+1}} a_{ux,w} \otimes a_{w,vy} \\
			&= \sum_{y\in X} b_{ux,vy}.
		\end{align*}
		So by the universal property of $\mathbb{A}_X$ there is a homomorphism $\Delta : \mathbb{A}_X \to \mathbb{A}_X \otimes \mathbb{A}_X$ such that
		\[
		\Delta(a_{u,v}) = \sum_{w\in X^n} a_{u,w} \otimes a_{w,v}.
		\]
		For coassociativity, we have
		\begin{align*}
			(\id\otimes \Delta)\circ\Delta(a_{u,v}) &= \sum_{w\in X^n} a_{u,w} \otimes \Delta(a_{w,v}) \\
			&= \sum_{w\in X^n} a_{u,w} \otimes \left(\sum_{z\in X^n} a_{w,z} \otimes a_{z,v}\right) \\
			&= \sum_{z\in X^n} \left(\sum_{w\in X^n} a_{u,w} \otimes a_{w,z} \right) \otimes a_{z,v} \\
			&= \sum_{z\in X^n} \Delta(a_{u,z}) \otimes a_{z,v} \\
			&= (\Delta\otimes \id)\circ\Delta(a_{u,v}).
		\end{align*}
		
Finally, we show that that the set of matrices
\[
 \{a_n=(a_{u,v})_{u,v\in X^n} \in M_{X^n}(\A_X) : n\geq 1\}
\]
satisfy the conditions of Definition \ref{CQG:matrixdefn}. Conditions (\ref{CQG:matrixdef-comult}) and (\ref{CQG:matrixdef-dense}) are clear. For (\ref{CQG:matrixdef-inverse}) we show that given any $n\geq 1$ the matrix $a_n$ is invertible with inverse given by $(a_n)^T$. Given $u,v\in X^n$ we have
\[
 (a_n(a_n)^T)_{u,v} = \sum_{w\in X^n} a_{u,w}a_{v,w} = \delta_{u,v} \sum_{w\in X^n} a_{u,w} = \delta_{u,v} 1_A .
\]
Likewise, we can show $((a_n)^T a_n)_{u,v} = \delta_{u,v}1_A$ and hence $(a_n)^T = a_n^{-1}$ as required.
	\end{proof}

\begin{remark}
	The canonical dense $*$-subalgebra of $\A_X$ is the $*$-subalgebra generated by the projections $\{a_{u,v}:u,v\in X^n, n\ge 0\}$. This is a Hopf $*$-algebra with counit $\varepsilon\colon\A_X\to\C$ and coinverse $\kappa\colon\A_X\to\A_X$ satisfying $\varepsilon(a_{u,v})=\delta_{u,v}$ and $\kappa(a_{u,v})=a_{v,u}$, for $u,v\in X^n$, $n\in\N$.
\end{remark}
	
We now show that $(\mathbb{A}_X,\Delta)$ is the quantum automorphism group (in the sense of \cite[Definition~2.3]{Wang1998}) of the homogeneous rooted tree.
	
\begin{prop} \label{prop:treeauts}
There is an action $\gamma = (\gamma_n)_{n=1}^\infty$ of $\mathbb{A}_X$ on $X^*$. Moreover, if $\alpha = (\alpha_n)_{n=1}^\infty$ is an action of a compact quantum group $(A,\Phi)$ on  $X^*$ then there is a quantum group homomorphism $\pi : \mathbb{A}_X \to A$ such that $(\pi\otimes\id)\circ\gamma_n = \alpha_n$ for any $n\geq 1$.
\end{prop}

\begin{proof}
For any $n\geq 1$, the elements
\[
 q_w := \sum_{w' \in X^n} a_{w,w'} \otimes p_{w'} \in \mathbb{A}_X \otimes C(X^n)
\]
for each $w\in X^n$ are mutually orthogonal projections and satisfy
\[
 \sum_{w\in X^n} q_w = \sum_{w,w'\in X^n} a_{w,w'} \otimes p_{w'} = 1\otimes 1
\]

Therefore there is a unital $*$-homomorphism $\gamma_n : C(X^n) \to \mathbb{A}_X \otimes C(X^n)$ satisfying $\gamma_n(p_w) = q_w$.
We have
\begin{align*}
(\Delta\otimes\id)\gamma_n(p_w) &= \sum_{w'\in X^n} \Delta(a_{w,w'}) \otimes p_{w'} \\
&= \sum_{w',z\in X^n} a_{w,z} \otimes a_{z,w'} \otimes p_w' \\
&= \sum_{z\in X^n} a_{w,z} \otimes \alpha_n(p_z) \\
&= (\id\otimes \gamma_n)\gamma_n(p_w),
\end{align*}
and so each $\gamma_n$ satisfies the coaction identity.

For a fixed $v\in X^n$ we have
\[
 \sum_{u\in X^n} \gamma_n(p_u)(a_{u,v}\otimes 1) = \sum_{u,w\in X^n}a_{u,w}a_{u,v} \otimes p_w = \sum_{u\in X^n} a_{u,v} \otimes p_v = 1 \otimes p_v.
\]
Multiplying by any element $a \otimes 1 \in \mathbb{A}_X \otimes 1$ shows that $\gamma_n(C(X^n))(\mathbb{A}_X\otimes 1)$ contains the elements $a \otimes p_v$ of $\mathbb{A}_X \otimes C(X^n)$ and hence the required density is satisfied.

Finally, fix $m < n$ and $w\in X^m$. Then
\begin{align*}
(\id\otimes i_{m,n})\gamma_m(p_w) &= \sum_{z\in X^m} a_{w,z} \otimes i_{m,n}(p_z) \\
&= \sum_{z\in X^m} \sum_{z'\in X^{n-m}} a_{w,z} \otimes p_{zz'} \\
&= \sum_{z\in X^m} \sum_{z'\in X^{n-m}} \sum_{w'\in X^{n-m}} a_{ww',zz'} \otimes p_{zz'} \\
&= \sum_{w'\in X^{n-m}} \alpha_n(p_{ww'}) \\
&= \gamma_n(i_{m,n}(p_w)),
\end{align*}
and so the collection $\gamma = (\gamma_n)_{n=1}^\infty$ defines an action of $(\mathbb{A}_X,\Delta)$ on the homogeneous rooted tree $X^*$.

Now suppose $(\alpha_n)_{n=1}^\infty$ is an action of a compact quantum group $(A,\Phi)$ on $X^*$. Let $b_{\varnothing,\varnothing} := 1 \in A$ and for $n \geq 1$ and $u,v\in X^n$ define $b_{u,v} \in A$ to be the unique elements satisfying
\[
 \alpha_n(p_u) = \sum_{v\in X^n} b_{u,v} \otimes p_v.
\]
The coaction identity for $\alpha_n$ says that
\begin{equation}
  \Phi(b_{u,v}) = \sum_{w\in X^n} b_{u,w} \otimes b_{w,v} \label{eq:comultforb}
\end{equation} 
for any $u,v\in X^n$.

We claim that the collection $\{b_{u,v}: u,v\in X^n, n\geq 0\} \subseteq A$ satisfies Definition \ref{def:qaut}. Condition (\ref{qaut:identity}) is by definition. For (\ref{qaut:proj}) and (\ref{qaut:sum}), we appeal to the universal property of the quantum permutation groups $A_s(|X|^n)$ for $n\geq 1$. Since for any $n\geq 1$, $\alpha_n$ defines a coaction of $(A,\Phi)$ on $C(X^n)$, \cite[Theorem 3.1]{Wang1998} says that the elements $\{b_{u,v} : u,v\in X^n\}$ satisfy conditions (3.1)--(3.3) of \cite[Section 3]{Wang1998}. Condition (3.1) is precisely (\ref{qaut:proj}). Conditions (3.1) and (3.2) say that for any $v\in X^n$ we have
\[
 \sum_{u\in X^n} b_{u,v} = 1_A = \sum_{w\in X^n} b_{v,w}.
\]
For any $u\in X^n$ and $x\in X$ we have
\[
 p_{ux} \leq \sum_{y\in X} p_{uy} = i_{n,n+1}(p_u),
\]
and hence
\begin{align*}
 \sum_{v\in X^n} \sum_{y\in X} b_{ux,vy}\otimes p_{vy} &= \alpha_{n+1}(p_{ux}) \\
 &\leq \alpha_{n+1}(i_{n,n+1}(p_u)) \\
 &= (\id_A\otimes i_{n,n+1})\alpha_n(p_u) \\
 &= \sum_{v\in X^n} \sum_{y\in X} b_{u,v}\otimes p_{vy}.
\end{align*}
It follows that $b_{ux,vy} \leq b_{u,v}$ for any $x,y\in X$. Therefore, for any $u,v\in X^n$ and $x\in X$ we have
\[
 b_{u,v} = b_{u,v}\left( \sum_{w\in X^n}\sum_{y\in X} b_{ux,wy} \right)= \sum_{y\in X} b_{ux,vy}.
\]
Likewise for any $y\in X$ we have $b_{u,v} = \sum_{x\in X} b_{ux,vy}$ and (\ref{qaut:sum}) holds.

Therefore, the universal property of $\mathbb{A}_X$ provides a homomorphism $\pi : \mathbb{A}_X \to A$ satisfying $\pi(a_{u,v}) = b_{u,v}$. It follows from \eqref{eq:comultforb} that $(\pi\otimes\pi)\circ\Delta = \Phi\otimes\pi$ and so $\pi$ is a compact quantum group homomorphism. The identity $(\pi\otimes\id)\circ\gamma_n = \alpha_n$ is immediate.
\end{proof}
	
\begin{prop} For $|X| \geq 2$ the $C^*$-algebra $\mathbb{A}_X$ is non-commutative and infinite dimensional.
\end{prop}
	
	\begin{proof} Without loss of generality, assume $X=\{0,1\}$. Let $B$ be the universal unital $C^*$-algebra generated by two (non-commuting) projections $p$ and $q$. It is known from \cite{RaeburnSinclair89} that $B \cong C^*(\mathbb{Z}_2*\mathbb{Z}_2)$, which is non-commutative and infinite dimensional. Define the matrix
\[
(b_{u,v})_{u,v\in X^2} = 
\begin{pmatrix}
p & 1_B-p & 0 & 0 \\
1_B-p & p & 0 & 0 \\ 
0 & 0 & q & 1_B-q \\
0 & 0 & 1_B-q & q \\
\end{pmatrix} \in M_4(B).
\]
Define $b_{\varnothing,\varnothing} = b_{0,0} = b_{1,1} = 1_B$, $b_{0,1} = b_{1,0} = 0$ and for $u,v\in X^2$ and $w,w' \in X^*$ define $b_{uw,vw'}:= \delta_{w,w'}b_{u,v}$. Then these elements satisfy the relations in Definition~\ref{def:qaut} and hence there is a surjective homomorphism $\A_X \to B$. Since $B$ is non-commutative and infinite-dimensional so is $\A_X$.

\end{proof}
	
\begin{remark} \label{rmk:classicalautgroup} The group $\Aut(X^*)$ of automorphisms of a homogeneous rooted tree $X^*$ is compact totally disconnected Hausdorff group under the permutation topology. A neighbourhood basis of the identity is given by the family of subgroups 
	\[
	\{G_u:=\{g\in\Aut(X^*): g\cdot u=u\}: u\in X^*\},
	\]
	and since the orbit of any $u\in X^*$ is finite, each of these open subgroups is closed and hence compact. Cosets of these subgroups are of the form $G_{u,v}:=\{g\in G: g\cdot v=u\}$. Then $\{G_{u,v}:u,v\in X^*\}$ is a basis of compact open sets for the topology on $\Aut(X^*)$. It follows that  the indicator functions $f_{u,v} := 1_{G_{u,v}}$ span a dense subset of $C(\Aut(X^*))$. It is easily checked that the elements $f_{u,v}$ satisfy \eqref{qaut:identity}--\eqref{qaut:sum} of Definition \ref{def:qaut} and the universal property of $C(\Aut(X^*))$ then implies that it is the abelianisation of $\A_X$.
\end{remark}

	\section{Self-similarity}\label{sec:selfsimilarity}
	If $g\in \Aut(X^*)$ and $x\in X$, the \textit{restriction} $g|_x$ is the unique element of $\Aut(X^*)$ satisfying
	\[
	g\cdot (xw) = (g\cdot x)g|_x\cdot w\quad\text{for all $w\in X^*$}.
	\]
A subgroup $G \leq \Aut(X^*)$ is called \emph{self-similar} if $G$ is closed under taking restrictions. That is, whenever $g\in G$ and $x\in X$, the restriction $g|_x$ is an element of $G$. With the topology inherited from $\Aut(X^*)$, the restriction map $G\to G\colon g\mapsto g|_x$ is continuous. If $G$ is any group acting on $X^*$ by automorphisms, we call the action \emph{self-similar} if the image of $G$ in $\Aut(X^*)$ is self-similar.

To have a reasonable notion of self-similarity for quantum subgroups of $\mathbb{A}_X$, we need to understand how restriction manifests itself in the function algebra $C(\Aut(X^*))$. Given $x\in X$ and $u,v\in X^n$ we have
\begin{align*}
	\{g:g|_x\cdot u = v\} &= \left(\bigcup_{y\in X} \{g:g\cdot x = y\}\right) \cap \{g:g|_x\cdot u = v\} \\
	&= \bigcup_{y\in X} \left(\{g:g\cdot x = y\}\cap \{g:g|_x\cdot u = v\}\right) \\
	&= \bigcup_{y\in X} \{g:g\cdot (xu) = yv\},
\end{align*}
and hence the corresponding indicator functions satisfy
\[
1_{\{g:g|_x\cdot u = v\}} = \sum_{y\in X} 1_{\{g:g\cdot (xu) = yv\}}.
\]
This formula motivates the following result.


\begin{prop}\label{prop:restrictions}
	For each $x\in X$ there is a homomorphism $\rho_x : \mathbb{A}_X \to \mathbb{A}_X$ satisfying
	\begin{equation} \label{restriction}
		\rho_x(a_{u,v}) = \sum_{y\in X} a_{yu,xv},
	\end{equation}
	for all $u,v\in X^n$.
\end{prop}

We illustrate the formula for a restriction map in Figure 1 by considering $X=\{0,1,2\}$ and looking at what the restriction map $\rho_1$ does to the projection $a_{1,2}$. 

\begin{figure}[H]
	\scalebox{0.75}{\begin{tikzpicture}[declare function={boxW=1.3cm;},box/.style={minimum size=boxW,draw}]  
			
			\foreach \x in {0,...,2}
			{
				\foreach \y in {0,...,2}
				{
					\ifnum\x=2
					\ifnum\y=1
					\node [box,fill=black] at ({\x*boxW}, {-\y*boxW}) {};
					\else
					\node [box] at ({\x*boxW}, {-\y*boxW}) {};
					\fi
					\else
					\node [box] at ({\x*boxW}, {-\y*boxW}) {};
					\fi
				}
			}
			\node [box, line width=1pt, minimum size = 3*boxW] at (boxW,-boxW) {};
		\end{tikzpicture}
	}
\raisebox{8ex}{\qquad$\mapsto$\qquad}
	\scalebox{0.25}{\begin{tikzpicture}[declare function={boxW=1.3cm;},box/.style={minimum size=boxW,draw}]  
			
			\foreach \x in {0,...,8}
			{
				\foreach \y in {0,...,8}
				{
					\node [box] at ({\x*boxW}, {-\y*boxW}) {};
				}
			}
			\node [box,fill=black] at ({5*boxW}, {-1*boxW}) {};
			\node [box,fill=black] at ({5*boxW}, {-4*boxW}) {};
			\node [box,fill=black] at ({5*boxW}, {-7*boxW}) {};
			
			\node [box, line width = 3pt,minimum size = 3*boxW] at (boxW,-boxW) {};
			\node [box, line width = 3pt,minimum size = 3*boxW] at (boxW,-4*boxW) {};
			\node [box, line width = 3pt,minimum size = 3*boxW] at (boxW,-7*boxW) {};
			\node [box, line width = 3pt,minimum size = 3*boxW] at (4*boxW,-boxW) {};
			\node [box, line width = 3pt,minimum size = 3*boxW] at (4*boxW,-4*boxW) {};
			\node [box, line width = 3pt,minimum size = 3*boxW] at (4*boxW,-7*boxW) {};
			\node [box, line width = 3pt,minimum size = 3*boxW] at (7*boxW,-boxW) {};
			\node [box, line width = 3pt,minimum size = 3*boxW] at (7*boxW,-4*boxW) {};
			\node [box, line width = 3pt,minimum size = 3*boxW] at (7*boxW,-7*boxW) {};
		\end{tikzpicture}
	}
	\caption{$\rho_1(a_{1,2}) = a_{01,12}+a_{11,12}+a_{21,12}$}
\end{figure}
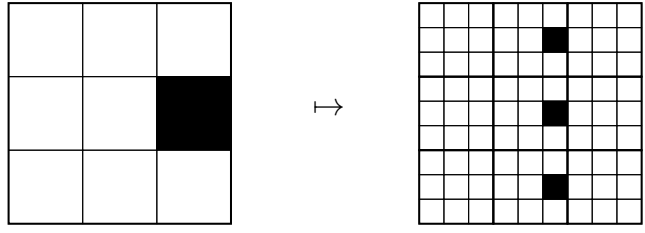

\begin{proof}[Proof of Proposition~\ref{prop:restrictions}]
	Fix $x\in X$. We show that the elements 
	\[
	\{b_{u,v}:=\rho_x(a_{u,v}) : u,v\in X^n, n\geq 1\}
	\]
	satisfy the conditions of Definition \ref{def:qaut}. For (\ref{qaut:identity}) we have
	\[
	b_{\varnothing,\varnothing} = \rho_x(a_{\varnothing,\varnothing}) = \sum_{y\in X} a_{y,x} = 1.
	\]
	For (\ref{qaut:proj}), we have
	\[
	b_{u,v}^* = \left(\sum_{y\in X} a_{yu,xv}\right)^* = \sum_{y\in X} a_{yu,xv}^* = \sum_{y\in X} a_{yu,xv} = b_{u,v}
	\]
	and
	\begin{align*}
		b_{u,v}^2 = \left(\sum_{y\in X} a_{yu,xv}\right)^2 = \sum_{y,z\in X} a_{yu,xv}a_{zu,xv} = \left(\sum_{y\in X} a_{yu,xv}\right) = b_{u,v}.
	\end{align*}
	For (\ref{qaut:sum}), fix $y\in X$. Then
	\[
	\sum_{z\in X} b_{uy,vz} = \sum_{z\in X} \sum_{w\in X} a_{wuy,xvz} = \sum_{w\in X} a_{wu,xv} = b_{u,v}.
	\]
	A similar calculation shows $\sum_{z\in X} b_{uz,vy} = b_{u,v}$. Hence there is a homomorphism $\rho_x$ with the desired formula. 
\end{proof}

\begin{remark}
	We define $\rho_\varnothing$ to be the identity homomorphism $\A_X\to\A_X$, and for $w=w_1\cdots w_n\in X^n$ we define $\rho_w$ to be the composition $\rho_{w_1}\circ \cdots\circ \rho_{w_n}$. A routine calculation shows that for all $u,v\in X^n$ we have 
	\[
	\rho_w(a_{u,v})=\sum_{z\in X^n}a_{zu,wv}.
	\]
\end{remark}

\begin{remark}\label{rem:thesigmamap}
	A similar argument to the one in the proof of Proposition~\ref{prop:restrictions} shows that for each $x\in X$ there is a homomorphism $\sigma_x\colon \A_X\to \A_X$ satisfying
	\[
	\sigma_x(a_{u,v})=\sum_{y\in X}a_{xu,yv}
	\]
	for all $u,v\in X^n$, $n\in \N$. It is straightforward to see that $\sigma_x=\kappa\circ\rho_x\circ \kappa$, where $\kappa$ is the coinverse. 
\end{remark}

We can now state the main definition of the paper.

\begin{definition}
	We call $\rho_w$ the \textit{restriction by} $w$. A quantum subgroup $A$ of $\mathbb{A}_X$ is \emph{self-similar} if for each $x\in X$ the restriction $\rho_x$ factors through the quotient map $q:\mathbb{A}_X \to A$; that is, if there exists a homomorphism $\widetilde{\rho_x} : A \to A$ such that the diagram
	\begin{center}
		\begin{tikzcd}
			\mathbb{A}_X \arrow[r, "\rho_x"] \arrow[d, "q"]
			& \mathbb{A}_X \arrow[d, "q"] \\
			A \arrow[r, "\widetilde{\rho_x}"]
			& A
		\end{tikzcd}
	\end{center}
	commutes.
\end{definition}

To motivate the main result of this section, let $G$ be a group. To construct a self-similar action of $G$ on $X^*$, it suffices to have a function $f\colon G\times X\to X\times G$ such that $f(e,x)=(x,e)$ for all $x\in X$, and such that the following diagram commutes
\begin{center}
	\begin{tikzcd}
		G\times G\times X \arrow[r, "m_G\times \id_X"] \arrow[d, "\id_G\times f"]
		& G\times X \arrow[dd, "f"] \\
		G\times X\times G \arrow[d, "f\times\id_G"] & \\
		X\times G\times G \arrow[r, "\id_X\times m_G"] & X\times G
	\end{tikzcd}
\end{center}
This data allows us to define an action of $G$ on $X^*$, which is self-similar with $g\cdot x$ and $g|_x$ the unique elements of $X$ and $G$ satisfying $(g\cdot x,g|_x):=f(g,x)$.

Our next result is a compact quantum group analogue of the above result. We will be working with multiple different identity homomorphisms and units. For clarity we adopt the following notational conventions: we write $\id_A$ for the identity homomorphism on a $C^*$-algebra $A$, and for $n\geq 1$ write $\id_n$ for the identity homomorphism on the commutative $C^*$-algebra $C(X^n)$. Likewise, $1_A$ will denote the unit of $A$, $1$ and $1_n$ will denote the units of $C(X)$ and $C(X^n)$ respectively.

\begin{thm} \label{thm:psi}
	Suppose $(A,\Phi)$ is a compact quantum group equipped with a unital $*$-homomorphism $\psi : C(X) \otimes A \to A \otimes C(X)$ satisfying
	\begin{equation} \label{eq:actionrestrictionhom}
		(\Phi\otimes\id_1)\psi = (\id_A\otimes\psi)(\psi\otimes\id_A)(\id_1\otimes\Phi)
	\end{equation}
	and
	\begin{equation} \label{eq:actionrestrictionhomdensity}
		\overline{\psi(C(X)\otimes 1_A)(A\otimes 1)} = A\otimes C(X).
	\end{equation}
	Then $(A,\Phi)$ acts on the homogeneous rooted tree $X^*$ and moreover the image of $\mathbb{A}_X$, under the homomorphism $\pi : \mathbb{A}_X \to A$ from Proposition \ref{prop:treeauts}, is a self-similar compact quantum group.
\end{thm}

\begin{proof}
	We begin by defining an action of $(A,\Phi)$ on $X^*$. Identify $C(X)$ with $C(X)\otimes 1_A \subseteq C(X) \otimes A$ and let $\alpha_1 := \psi|_{C(X)\otimes 1_A}$. Then $\alpha_1$ is clearly unital and the coaction identity and Podle\' s condition for $\alpha_1$ follow from \eqref{eq:actionrestrictionhom} and \eqref{eq:actionrestrictionhomdensity}. Now inductively define $\alpha_{n+1} := (\psi\otimes\id_n)(\id_1\otimes\alpha_n) : C(X^{n+1}) \to A \otimes C(X^{n+1})$ for $n\geq 1$, where we are supressing the canonical isomorphism $C(X^{n+1}) \cong C(X)\otimes C(X^n)$. Again, $\alpha_{n+1}$ is clearly unital whenever $\alpha_n$ is. If we assume $\alpha_n$ satisfies the coaction identity, then
	\begin{align*}
		(\Phi\otimes\id_{n+1})\alpha_{n+1} &= (\Phi\otimes\id_{n+1})(\psi\otimes\id_n)(\id_1\otimes\alpha_n) \\
		&= (\id_A\otimes\psi\otimes\id_n)(\psi\otimes\id_A\otimes\id_n)(\id_1\otimes\Phi\otimes\id_n)(\id_1\otimes\alpha_n) \\
		&= (\id_A\otimes\psi\otimes\id_n)(\psi\otimes\id_A\otimes\id_n)(\id_1\otimes\id_A\otimes\alpha_n)(\id_1\otimes\alpha_n) \\
		&= (\id_A\otimes\psi\otimes\id_n)(\id_A\otimes\id_1\otimes\alpha_n)(\psi\otimes\id_n)(\id_1\otimes\alpha_n) \\
		&= (\id_A\otimes\alpha_{n+1})\alpha_{n+1},
	\end{align*}
	and so $\alpha_{n+1}$ also satisfies the coaction identity. Since $\alpha_1$ is a coaction, we see that $\alpha_n$ satisfies the coaction identity for any $n\geq 1$. 
	
	To see that each $\alpha_n$ satisfies the Podle\' s condition, we argue by induction. We know it is satisfied for $n=1$. Suppose for some $n\geq 1$ that
	\[
	\overline{\alpha_n(C(X^n))(A\otimes 1_n)} = A \otimes C(X^n).
	\]
	Fix a spanning element $a\otimes p_u \otimes p_x \in A \otimes C(X^{n+1})$ where $u\in X^n$ and $x\in X$. By the inductive hypothesis we can approximate
	\[
	a\otimes p_u \sim \sum_i \alpha_n(f_i)(a_i\otimes 1_n),
	\]
	where $f_i \in C(X^n)$ and $a_i\in A$. Then
	\begin{equation} \label{eq:actiondensityapprox}
		a\otimes p_u\otimes p_x \sim \sum_i (\alpha_n(f_i)\otimes 1)(1_A\otimes 1_n \otimes p_x)(a_i\otimes 1_{n+1}).
	\end{equation}
	By definition of $\alpha_n$, for any $f\in C(X^n)$ we have
	\begin{align*}
		\alpha_n(f)\otimes 1 &= ((\psi\otimes\id_{n-1})\dots(\id_{n-1}\otimes\psi)(f\otimes 1_A))\otimes 1 \\
		&= (\psi\otimes\id_n)\dots(\id_{n-1}\otimes\psi\otimes \id_1)(f\otimes 1_A\otimes 1) \\
		&= (\psi\otimes\id_n)\dots(\id_{n-1}\otimes\psi\otimes \id_1)(\id_n\otimes\psi)(f\otimes 1\otimes 1_A) \\
		&= \alpha_{n+1}(f\otimes 1).
	\end{align*}
	So we can write \eqref{eq:actiondensityapprox} as
	\[
	\sum_i \alpha_{n+1}(f_i\otimes 1)(1_A\otimes 1_n \otimes p_x)(a_i\otimes 1_{n+1}).
	\]
	Since $\psi$ is unital, we have
	\[
	1_A\otimes 1_n \otimes p_x = (\psi\otimes\id_n)(1 \otimes 1_A \otimes 1_{n-1}\otimes p_x),
	\]
	which can be approximated using the induction hypothesis by
	\begin{align*}
		(\psi\otimes\id_n)(1 \otimes 1_A \otimes 1_{n-1}\otimes p_x) &\sim (\psi\otimes\id_n)\left(1\otimes\sum_j \alpha_n(g_j)(b_j\otimes 1_n)\right) \\
		&= (\psi\otimes\id_n)\left(\sum_j(\id_1\otimes \alpha_n)(1\otimes g_j)(1\otimes b_j\otimes 1_n)\right) \\
		&= \sum_j \alpha_{n+1}(1\otimes g_j) (\psi\otimes \id_n)(1\otimes b_j\otimes 1_n).
	\end{align*}
	Finally, applying the Podle\' s condition for $\alpha_1$ we can approximate 
	\[
	\psi(1\otimes b_j) \sim \sum_k \alpha_1(h_k)(c_k\otimes 1) = \sum_k \psi(h_k\otimes 1_A)(c_k\otimes 1),
	\]
	so
	\begin{align*}
		(\psi\otimes \id_n)(1\otimes b_j\otimes 1_n) &\sim \sum_k (\psi(h_k\otimes 1_A)\otimes 1_n)(c_k\otimes 1_{n+1}) \\
		&= \sum_k \alpha_{n+1}(h_k\otimes 1_n)(c_k\otimes 1_{n+1}).
	\end{align*}
	Combining these approximations we can write
	\[
	a\otimes p_u\otimes p_x \sim \sum_{i,j,k} \alpha_{n+1}((f_i\otimes 1)(h_k\otimes g_j))(c_k a_i \otimes 1_{n+1}),
	\]
	where $f_i,g_j \in C(X^n), h_k \in C(X)$ and $a_i, c_k \in A$. Thus $\alpha_{n+1}$ satisfies the Podle\' s condition and so by induction $\alpha_n$ satisfies the Podle\' s condition for every $n\geq 1$. 
	
	It remains to show that $\alpha_n\circ i_{m,n} = (\id_A\otimes i_{m,n})\circ\alpha_m$ for any $m<n$. As in the proof of Proposition \ref{prop:treeauts} , for any $n \geq 1$ and $u,v\in X^n$ we will let $b_{u,v} \in A$ be the unique elements satifsying
	\[
	\alpha_n(p_u) = \sum_{v\in X^n} b_{u,v}\otimes p_v.
	\]
	We know from the same proof that for any $n\geq 1$ and $v\in X^n$ we have
	\[
	\sum_{u\in X^n} b_{u,v} = 1_A.
	\]
	If $m < n$, for any $u\in X^m$ we have
	
	\begin{align*}
		\alpha_n\circ i_{m,n}(p_u) &= (\psi\otimes\id_{n-1})\dots(\id_{m-1}\otimes\psi\otimes\id_{n-m})(\id_m\otimes\alpha_{n-m})(i_{m,n}(p_u))\\
		&= \sum_{w\in X^{n-m}}(\psi\otimes\id_{n-1})\dots(\id_{m-1}\otimes\psi\otimes\id_{n-m})(\id_m\otimes\alpha_{n-m})(p_u\otimes p_w) \\
		&= \sum_{w,z\in X^{n-m}}(\psi\otimes\id_{n-1})\dots(\id_{m-1}\otimes\psi\otimes\id_{n-m})(p_u\otimes b_{w,z}\otimes p_z) \\
		&= \sum_{z\in X^{n-m}}(\psi\otimes\id_{m-1})\dots(\id_{m-1}\otimes\psi)\left(p_u\otimes \sum_{w\in X^{n-m}} b_{w,z}\right)\otimes p_z \\
		&= \sum_{z\in X^{n-m}}(\psi\otimes\id_{m-1})\dots(\id_{m-1}\otimes\psi)(p_u\otimes 1_A )\otimes p_z \\
		&= \sum_{z\in X^{n-m}}\alpha_m(p_u)\otimes p_z \\
		&= \sum_{v\in X^m} b_{u,v} \otimes \sum_{z\in X^{n-m}} p_v\otimes p_z \\
		&= \sum_{v\in X^m} b_{u,v} \otimes i_{m,n}(p_v) \\
		&= (\id_A\otimes i_{m,n})\circ\alpha_m(p_u).
	\end{align*}
	So we have that $(\alpha_n)_{n=1}^\infty$ defines an action of $(A,\Phi)$ on $X^*$.
	
	Finally, let $\pi : \A_X \to A$ be the homomorphism from Proposition \ref{prop:treeauts}. We have $\pi(a_{u,v}) = b_{u,v}$ for any $u,v\in X^n$ and $n\geq 1$. For each $x\in X$ define a homomorphism $\tilde{\rho_x} : A\to A$ by
	\[
	\psi(1\otimes a) = \sum_{x\in X} \tilde{\rho_x}(a) \otimes p_x,
	\]
	where $a\in A$. For any $u\in X^n$ we have
	\[
	\alpha_{n+1}(1\otimes p_u) = \sum_{y\in X} \alpha_{n+1}(p_{yu}) = \sum_{v\in X^n} \sum_{x,y\in X} b_{yu,xv} \otimes p_x \otimes p_v.
	\]
	On the other hand, we know $\alpha_{n+1} = (\psi\otimes \id_n)(\id_1\otimes \alpha_n)$ and
	\[
	(\psi\otimes \id_n)(\id_1\otimes \alpha_n)(1\otimes p_u) = \sum_{v\in X^n} \psi(1\otimes b_{u,v}) \otimes p_v = \sum_{v\in X^n} \sum_{x\in X} \tilde{\rho_x}(b_{u,v})\otimes p_x \otimes p_v,
	\]
	and by comparing tensor factors we see that $\tilde{\rho_x}(b_{u,v}) = \sum_{y\in X} b_{yu,xv}$. Hence, the diagram
	\begin{center}
		\begin{tikzcd}
			\mathbb{A}_X \arrow[r, "\rho_x"] \arrow[d, "\pi"]
			& \mathbb{A}_X \arrow[d, "\pi"] \\
			\pi(\A_X) \arrow[r, "\tilde{\rho_x}"]
			& \pi(\A_X)
		\end{tikzcd}
	\end{center}
	commutes, and so $\pi(\A_X) \subseteq A$ is a self-similar quantum group.
\end{proof}

\begin{prop}
	The following are equivalent
	\begin{enumerate}
		\item $(A,\Phi)$ is a quantum self-similar group, and
		\item $(A,\Phi)$ is a quantum subgroup of $(\A_X,\Delta)$ and there is a homomorphism $\psi : C(X) \otimes A \to A \otimes C(X)$ satisfying the hypotheses of Theorem~\ref{thm:psi}.
	\end{enumerate}
\end{prop}

\begin{proof}
	Theorem~\ref{thm:psi} is the implication (2) $\Longrightarrow$ (1). To see (1) $\Longrightarrow$ (2) suppose $(A,\Phi)$ is a quantum self-similar group. By definition there is a surjective quantum group morphism $q:\A_X \to A$. It is routine to check that there is a homomorphism $\psi : C(X)\otimes A \to A \otimes C(X)$ satisfying
	\[
	\psi(p_x\otimes q(a_{u,v})) = \sum_{y\in X} q(a_{xu,yv}) \otimes p_y.
	\]
	Given $u,v \in X^n$ we have
	\begin{align*}
		(\Phi\otimes\id_1)\psi(p_x\otimes q(a_{u,v})) &= \sum_{y\in X} \Phi(q(a_{xu,yv})) \otimes p_y \\
		&= \sum_{w\in X^n} \sum_{y,z \in X} q(a_{xu,zw}) \otimes q(a_{zw,yv}) \otimes p_y \\
		&= (\id_A\otimes\psi)\left( \sum_{w\in X^n} \sum_{z\in X} q(a_{xu,zw}) \otimes p_z \otimes q(a_{w,v}) \right) \\
		&= (\id_A\otimes\psi)(\psi\otimes\id_A)\left(\sum_{w\in X^n} p_x \otimes q(a_{u,w})\otimes q(a_{w,v}) \right)\\
		&= (\id_A\otimes\psi)(\psi\otimes\id_A)(\id_1\otimes\Phi)(p_x\otimes q(a_{u,v})),
	\end{align*}
	and so $\psi$ satisfies \eqref{eq:actionrestrictionhom}. For \eqref{eq:actionrestrictionhomdensity} notice that for any $q(a) \in A$ and $z\in X$ we have
	\begin{align*}
		q(a) \otimes p_z &= (1\otimes p_z)(q(a)\otimes 1) \\
		&= \left(\sum_{x\in X} q(a_{x,z})\otimes p_z) \right)(q(a)\otimes 1) \\
		&= \left(\sum_{x,w\in X}q(a_{x,w})\otimes p_w) (q(a_{x,z})\otimes 1)\right)(q(a)\otimes 1) \\
		&= \sum_{x\in X} \psi(p_x\otimes 1) (q(a_{x,z}a)\otimes 1).\qedhere
	\end{align*}
\end{proof}

\begin{example}\label{eg:commutativeSSG}
	If $G$ is a closed subgroup of $\Aut(X^*)$ which is self-similar, then $C(G)$ is a commutative self-similar quantum group. The quotient map $\A_X \to C(Aut(X^*))$ takes a generator $a_{u,v}$ to the indicator function $f_{u,v}$ defined in Remark \ref{rmk:classicalautgroup}. For a function $f\in C(\Aut(X^*))$ and $x\in X$ the restriction homomorphism $\tilde{\rho_x}$ satisfies $\tilde{\rho_x}(f)(g)=f(g|_x)$, for any $g\in G$.
\end{example}

\section{Finitely constrained self-similar quantum groups}\label{sec:finitelyconstrained}

\subsection{Classical finitely constrained self-similar groups}

Fix $d\geq 1$, and let $X^{[d]} = \bigcup_{k \leq d} X^k$ be the finite subtree of $X^*$ of depth $d$. The group of automorphism $\Aut(X^{[d]})$ is a quotient of $\Aut(X^*)$, and the quotient map is given by restriction to the finite subtree. We write $r_d : \Aut(X^*) \to \Aut(X^{[d]})$ for this restriction map.

Fix a subgroup $P \leq \Aut(X^{[d]})$. Define
\[
G_P := \{ g\in \Aut(X^*) : r_d(g|_w) \in P \mbox{ for all } w\in X^* \}.
\]
By the properties of restriction, if $g,h\in G_P$, then for any $w\in X^*$
\[
r_d((gh)|_w) = r_d(g|_{h\cdot w} h|_w) = r_d(g|_{h\cdot w}) r_d(h|_w) \in P.
\]
Likewise, $r_d(g^{-1}|_w) = r_d(g|_{g^{-1}\cdot w})^{-1} \in P$. Hence $G_P$ is a self-similar group, called a \emph{finitely constrained self-similar group}. More details for these groups can be found in \cite{Sunic11}.

\subsection{Finitely constrained self-similar quantum groups}

Consider the subalgebra $\mathbb{A}_d \subseteq \mathbb{A}_X$ generated by the elements $\{a_{u,v} : |u|=|v|\leq d\}$. Since $\Delta : \mathbb{A}_d \to \mathbb{A}_d\otimes \mathbb{A}_d$ the subalgebra $\mathbb{A}_d$ is a quotient quantum group. The abelianisation of $\mathbb{A}_d$ is the algebra $C(\Aut(X^{[d]})$ of continuous functions on the finite group $\Aut(X^{[d]})$.

\begin{definition}\label{def:FinitelyConstrainedCQG}
	Suppose $\mathbb{P}$ is a quantum subgroup of $\mathbb{A}_d$, where $\Pp=\A_d/I$. Denote by $q_I : \mathbb{A}_d \to \mathbb{P}$ the quotient map; so $I = \ker(q_I)$. We denote by $J$ the smallest closed 2-sided ideal of $\mathbb{A}_X$ generated by $\{\rho_w(I) : w\in X^*\}$, and by $A_\Pp$ the quotient $A_\Pp:=\A_X/J$. In the next result we prove that $A_\Pp$ is a self-similar quantum group, and we call it a \textit{finitely constrained self-similar quantum group}.
\end{definition}

\begin{prop}\label{prop:GPisaCQG}
	Each $A_\Pp$ is a self-similar quantum group.
\end{prop}

To prove Proposition~\ref{prop:GPisaCQG} we need two lemmas. Recall that for $g,h \in \Aut(X^*), w\in X^*$ we have
\[
(gh)|_w = g|_{h\cdot w} h|_w.
\]
In the first lemma, we establish an analogous relationship between the comultiplication $\Delta$ on $\A_X$ and the restriction maps $\rho_w$.

\begin{lemma}\label{lem:DeltaAndRho}
	For any $n\geq 1, w\in X^n$ and $a \in \mathbb{A}_X$ we have
	\[
	(\Delta\circ\rho_w)(a) = \sum_{y\in X^n} (1\otimes a_{y,w})(\rho_y\otimes\rho_w)(\Delta(a)).
	\]
\end{lemma}

\begin{proof}
	Let $a_{u,v}$ be a generator of $\mathbb{A}_X$, with $|u| = |v| = k \geq 0$. Then
	\begin{align*}
		(\Delta\circ\rho_w)(a_{u,v}) &= \Delta\left( \sum_{\alpha\in X^n} a_{\alpha u,wv} \right) \\
		&= \sum_{y\in X^n} \sum_{\beta\in X^k} \sum_{\alpha\in X^n} a_{\alpha u,y\beta} \otimes a_{y\beta,wv} \\
		&= \sum_{y\in X^n} \sum_{\beta\in X^k} \rho_y(a_{u,\beta}) \otimes a_{y\beta,wv} \\
		&= \sum_{y\in X^n} \sum_{\beta\in X^k} \rho_y(a_{u,\beta}) \otimes a_{y,w}\rho_w(a_{\beta,v}) \\
		&= \sum_{y\in X^n} (1\otimes a_{y,w}) (\rho_y \otimes \rho_w)\left(\sum_{\beta\in X^k} a_{u,\beta}\otimes a_{\beta, v}\right) \\
		&= \sum_{y\in X^n} (1\otimes a_{y,w}) (\rho_y \otimes \rho_w(\Delta(a_{u,v})).
	\end{align*}
	To see that this formula extends to $\mathbb{A}_X$ it's enough to show that for any $w\in X^*$ the map 
	\[
	a \mapsto \sum_{y\in X^n} (1\otimes a_{y,w})(\rho_y\otimes\rho_w)(\Delta(a))
	\]
	is linear and multiplicative. Linearity is clear, and multiplicativity follows from the orthogonality of the projections $1\otimes a_{y,w}$ and $1\otimes a_{z,w}$ for $y\neq z$.
\end{proof}

\begin{lemma}\label{lem:KernelsAndQuotients}
	Consider the quotient maps $q_I : \mathbb{A}_d \to \mathbb{A}_d/ I$ and $q_J : \mathbb{A}_X \to \mathbb{A}_X/ J$. Then for any $n\geq 1$ and $y,w \in X^n$
	\[
	\ker (q_I \otimes q_I) \subseteq \ker ((q_J\circ \rho_y) \otimes (q_J\circ \rho_w)).
	\]
\end{lemma}

\begin{proof}
	By definition of $J$ we have $I\subseteq J\circ\rho_w$ for any $w\in X^*$. Therefore there is a commuting diagram
	\begin{center}
		\begin{tikzcd}
			\mathbb{A}_d \arrow[r, hook] \arrow[d, "q_I"]
			& \mathbb{A}_X \arrow[d, "q_J\circ \rho_w"] \\
			\mathbb{A}_d/I \arrow[r, "\pi_w"]
			& \mathbb{A}_X/\ker(q_J\circ \rho_w)
		\end{tikzcd}
	\end{center}
	Then if $c \in \ker(q_I\otimes q_I)$ we have
	\[
	(q_J\circ \rho_y) \otimes (q_J\circ \rho_w)(c) = (\pi_y \circ q_I)\otimes (\pi_w\circ q_I) (c) = \pi_y\otimes\pi_w)\circ(q_I\otimes q_I)(c) =0,
	\]
	as required.
\end{proof}

\begin{proof}[Proof of Proposition~\ref{prop:GPisaCQG}]
	To see that $A_\Pp$ is a compact quantum group, it suffices to show that $J$ is a Woronowicz ideal. In other words, we need to show that $\Delta(J) \subset \ker (q_J\otimes q_J)$ where $q_J: \mathbb{A}_X \to \mathbb{A}_X / J =: A_\Pp$ is the quotient map. Since $J$ is generated as an ideal by $\bigcup_{w\in X^*} \rho_w(I)$ it's enough to show that
	\[
	(q_J\otimes q_J)(\Delta\circ\rho_w(i)) = 0
	\]
	for any $i\in I$ and $w\in X^*$. Because $I$ is a Woronowicz ideal we know that $\Delta(i) \in \ker(q_I\otimes q_I)$. Then by Lemmas~\ref{lem:DeltaAndRho}~and~\ref{lem:KernelsAndQuotients} we have
	\begin{align*}
		(q_J\otimes q_J)(\Delta\circ\rho_w(i)) & = (q_J\otimes q_J)\left( \sum_{y\in X^n} (1\otimes a_{y,w})(\rho_y\otimes\rho_w)(\Delta(i))\right)  \\
		&= \sum_{y\in X^n} (1\otimes q_J(a_{y,w}))(q_J\circ\rho_y\otimes q_J\circ\rho_w)(\Delta(i)) \\
		&=0.
	\end{align*}
	Finally, $A_\Pp$ is self-similar since by definition of $J$ we have $\rho_w(J) \subset J$ for any $w\in X^*$. 
\end{proof}

%
%
%
%

\subsection{Free wreath products}

It is well known that for any $d\ge 1$ the group $\Aut(X^{[d+1]})$ is isomorphic to the wreath product $\Aut(X^{[d]})\wr \Sym(X)$. Since $\Aut(X^*)$ is the inverse limit over $d$ of the groups $\Aut(X^{[d]})$, it can be thought as the infinitely iterated wreath product $\ldots\wr\Sym(X)\wr \Sym(X)$. It follows that $\Aut(X^*)\cong \Aut(X^*)\wr \Sym(X)$. More generally, it is shown in \cite{BS13} that if $P\le \Sym(X)=\Aut(X^{[1]})$, then the finitely constrained self-similar group $G_P$ is the infinitely iterated wreath product $\ldots \wr P\wr P$. In this section we prove in Theorem~\ref{thm:APWreath} an analogue of this result for finitely constrained self-similar quantum groups.

In \cite{Bichon04}, Bichon constructs a free wreath product of a compact quantum group by the quantum permutation group $\mathbb{A}_s(n)$. Bichon also comments in Remark~2.4 of \cite{Bichon04} that there is a natural analogue of this construction for free wreath products by quantum subgroups of $\mathbb{A}_s(n)$.  In this section we formally extend this definition to take free wreath products by any quantum subgroup of $\mathbb{A}_s(n)$, and we prove that the finitely constrained self-similar quantum group $A_\Pp$ induced from a quantum subgroup $\Pp$ of $A_s(n)$ is a free wreath product by $\Pp$. We begin by recalling the definition of the free wreath product from \cite{Bichon04}; note that we use our notation $\A_1$ instead of $A_s(|X|)$.

\begin{definition}\label{def:wreathproduct}
	Let $X$ be a set of at least two elements. Let $(A,\Phi)$ be a compact quantum group, and $\Pp$ a quantum subgroup of $\A_1$. For each $x\in X$, we denote by $\nu_x$ the inclusion of $A$ in the free product $C^*$-algebra $(\ast_{x\in X}A) \ast \Pp$. The \textit{free wreath product} of $A$ by $\Pp$ is the quotient of $(\ast_{x\in X}A) * \Pp$ by the two-sided ideal generated by the elements
	\[
	\nu_x(a)q_I(a_{x,y}) - q_I(a_{x,y})\nu_x(a), \ x,y\in X, a\in A.
	\]
	The resulting $C^*$-algebra is denoted by $A \ast_{X,w} \Pp$, and the quotient map is denoted by $q_w$. If $X$ is understood, we typically just write $A\ast_w\Pp$.
\end{definition}

\begin{thm}
	Let $(A,\Phi)$ be a compact quantum group, and $\Pp$ a quantum subgroup of $\A_1$. The free wreath product $A \ast_w \Pp$ from Definition~\ref{def:wreathproduct} is a compact quantum group with comultiplication $\Phi_w$ satisfying
	\begin{align}
		\label{eq:1forPhiw}\Phi_w(q_w(q_I(a_{x,y}))) &= \sum_{z\in X} q_w(q_I(a_{x,z}))\otimes q_w(q_I(a_{z,y}) )\\ 
		\label{eq:2forPhiw}\Phi_w(q_w(\nu_x(a))) &= \sum_{z\in X} (q_w\otimes q_w)\big((\nu_x \otimes \nu_z)(\Phi(a))(q_I(a_{x,z}) \otimes 1)\big), 
	\end{align}
	for each $x,y\in X$ and $a\in A$.
\end{thm}

\begin{proof}
	Since $I$ is a Woronowicz ideal, we have $\Delta|_I\subseteq \ker(q_I\otimes q_I)$, and so the map $(q_I\otimes q_I)\circ \Delta_{\A_1}$ descends to a map 
	\[
	\phi\colon \Pp\to \Pp\otimes\Pp\subseteq ((\ast_{x\in X}A) \ast \Pp)^{\otimes 2}.
	\] 
	Then $(q_w\otimes q_w)\circ \phi\colon \Pp\to (A\ast_w\Pp)^{\otimes 2}$ satisfies
	\[
	(q_w\otimes q_w)\circ \phi(q_I(a_{x,y}))=\sum_{z\in X}q_w(q_I(a_{x,z}))\otimes q_w(q_I(a_{z,y}))\quad\text{for all }x,y\in X.
	\]
	For each $x\in X$, consider the continuous linear map $\phi_x\colon A\to (A\ast_w\Pp)^{\otimes 2}$ given by
	\[
	\phi_x(a)=(q_w\otimes q_w)\left(\sum_{z\in X}(\nu_x\otimes\nu_z)(\Phi(a))(q_I(a_{x,z})\otimes 1)\right).
	\]
	We claim that $\phi_x$ is a homomorphism. To see this, let $\{a^\lambda = (a^\lambda_{i,j}) \in M_{d_\lambda}(A) : \lambda \in \Lambda\}$ is be a family of matrices satisfying (1)--(3) of Definition~\ref{CQG:matrixdefn}, and $\mathcal{A}$ be the $*$-subalgebra of $A$ spanned by the entries $a^\lambda_{i,j}$. Let $a,b\in \AA$ and use Sweedler's notation to write $\Phi(a)=a_{(1)}\otimes a_{(2)}$ and $\Phi(b)=b_{(1)}\otimes b_{(2)}$. We have 
	\[
		\phi_x(a)\phi_x(b)=\sum_{z,z'\in X}q_w\left(\nu_x(a_{(1)})q_I(a_{x,z})\nu_x(b_{(1)})q_I(a_{x,z'})\right)\otimes q_w(\nu_z(a_{(2)})\nu_{z'}(b_{(2)})),
	\]
	and then since
	\begin{align*}
		q_w\left(\nu_x(a_{(1)})q_I(a_{x,z})\nu_x(b_{(1)})q_I(a_{x,z'})\right)&=q_w\left(\nu_x(a_{(1)}b_{(1)})q_I(a_{x,z}a_{x,z'})\right)\\
		&= \delta_{z,z'}q_w\left(\nu_x(a_{(1)}b_{(1)})q_I(a_{x,z})\right),
	\end{align*}
	we have 
	\begin{align*}
		\phi_x(a)\phi_x(b)&=\sum_{z\in X}q_w\left(\nu_x(a_{(1)}b_{(1)})q_I(a_{x,z})\right)\otimes q_w(\nu_z(a_{(2)})\nu_{z}(b_{(2)}))\\
		&= (q_w\otimes q_w)\left(\sum_{z\in X}(\nu_x\otimes\nu_z)(a_{(1)}b_{(1)}\otimes a_{(2)}b_{(2)})(q_I(a_{x,z})\otimes 1)\right)\\
		&= (q_w\otimes q_w)\left(\sum_{z\in X}(\nu_x\otimes\nu_z)(\Phi(ab))(q_I(a_{x,z})\otimes 1)\right)\\
		&= \phi_x(ab).
	\end{align*}
	Since $\AA$ is dense in $A$, it follows that $\phi_x$ is a homomorphism on $A$. 
	
	The universal property of $(\ast_{x\in X}A) \ast \Pp$ now gives a homomorphism $\widetilde{\Phi}\colon (\ast_{x\in X}A) \ast \Pp\to (A\ast_w\Pp)^{\otimes 2}$ satisfying 
	\begin{align*}
		\widetilde{\Phi}(q_I(a_{x,y})) &= \sum_{z\in X} q_w(q_I(a_{x,z}))\otimes q_w(q_{\Pp}(a_{z,y}) ),\\
		\widetilde{\Phi}(\nu_x(a)) &= \sum_{z\in X} (q_w\otimes q_w)\big((\nu_x \otimes \nu_z)(\Phi(a))(q_I(a_{x,z}) \otimes 1)\big).
	\end{align*}
	For each $a\in \AA, x,y\in X$ we have
	\begin{align*}
		\widetilde{\Phi}(\nu_x(a)q_I(a_{x,y}))&= \sum_{z,z'\in X}q_w\left(\nu_x(a_{(1)})q_I(a_{x,z})q_I(a_{x,z'})\right)\otimes q_w(\nu_z(a_{(2)})q_I(a_{z',y}))\\
		&= \sum_{z,z'\in X}q_w\left(q_I(a_{x,z})\nu_x(a_{(1)})q_I(a_{x,z'})\right)\otimes q_w(q_I(a_{z',y})\nu_z(a_{(2)}))\\
		&= \widetilde{\Phi}(q_I(a_{x,y})\nu_x(a)).
	\end{align*}
	It follows that $\widetilde{\Phi}(\nu_x(a)q_I(a_{x,y}))=\widetilde{\Phi}(q_I(a_{x,y})\nu_x(a))$ for each $a\in A, x,y\in X$, and hence $\widetilde{\Phi}$ descends to the desired $\Phi_w\colon A\ast_w\Pp\to(A\ast_w\Pp)^{\otimes 2}$.
	
	We now claim that $(\id\otimes\Phi_w)\circ\Phi_w=(\Phi_w\otimes \id)\circ\Phi_w$. Since $\AA$ is dense in $A$, to see that $(\id\otimes\Phi_w)\circ\Phi_w$ and $(\Phi_w\otimes \id)\circ\Phi_w$ agree on each $q_w(\nu_x(A))$, it suffices to show that
	\[
	(\id\otimes\Phi_w)\circ\Phi_w(q_w(\nu_x(a^\lambda_{i,j})))=(\Phi_w\otimes \id)\circ\Phi_w(q_w(\nu_x(a^\lambda_{i,j}))),\,\text{for all }\lambda\in\Lambda, 1\le i,j \le d_\lambda.
	\]
	Routine calculations using \eqref{eq:1forPhiw}~and~\eqref{eq:2forPhiw} show that both sides of this equation are equal to
	\[
	\sum_{z,z'\in X}\sum_{1\le k,l\le d_\lambda}q_w^{\otimes 3}\big( \nu_x(a^\lambda_{i,k})q_I(a_{x,z})\otimes \nu_z(a^\lambda_{k,l})q_I(a_{z,z'})\otimes \nu_{z'}(a^\lambda_{l,j})\big),
	\]
	and hence are equal. So we have $(\id\otimes\Phi_w)\circ\Phi_w(q_q(\nu_x(A)))=(\Phi_w\otimes \id)\circ\Phi_w(q_q(\nu_x(A)))$ for each $x\in X$. It is straightforward to check that evaluating both $(\id\otimes\Phi_w)\circ\Phi_w$ and $(\Phi_w\otimes \id)\circ\Phi_w$ at $q_I(a_{x,y})$ gives
	\[
	\sum_{z,z'\in X}q_w^{\otimes 3}\left(q_I^{\otimes 3}\left(a_{x,z}\otimes a_{z,z'}\otimes a_{z',y}\right)\right).
	\] 
	Hence we have $(\id\otimes\Phi_w)\circ\Phi_w=(\Phi_w\otimes \id)\circ\Phi_w$.
	
	We now define the matrix $a^{X}$ by $a^X_{x,y}:=q_w(q_I(a_{x,y}))$, for $x,y\in X$; and for each $\lambda\in\Lambda$, $1\le i,j\le d_\lambda$, $x,y\in X$, the elements 
	\[
	a^{(\lambda,X)}_{(i,x),(j,y)}:=q_w(\nu_x(a^\lambda_{i,j})q_I(a_{x,y}))\in A \ast_w \Pp,
	\]
	define matrices $a^{(\lambda,X)}=(a^{(\lambda,X)}_{(i,x),(j,y)})$. To finish the proof we have to show that these matrices satisfy (1)--(3) of Definition~\ref{CQG:matrixdefn}.
	
	We have 
	\[
	\Phi_w(a^{(\lambda,X)}_{(i,x),(j,y)})=\Big(\sum_{z\in X}(q_w\otimes q_w)\big((\nu_x\otimes\nu_z)(\Phi(a^\lambda_{i,j}))(q_I(a_{x,z})\otimes 1)\big)\Big)(q_w\circ q_I)^{\otimes 2}(\Delta(a_{x,y})).
	\]
	We know that (1) is satisfied for the matrix $a^X$. For each $z\in X$ we have 
	\begin{align*}
		&(q_w\otimes q_w)\big((\nu_x\otimes\nu_z)(\Phi(a^\lambda_{i,j}))(q_I(a_{x,z})\otimes 1)\big)(q_w\circ q_I)^{\otimes 2}(\Delta(a_{x,y}))\\
		&\qquad= (q_w\otimes q_w)\Big((\nu_x\otimes\nu_z)\Big(\sum_{1\le k\le d_\lambda}a_{i,k}^\lambda\otimes a_{k,j}^\lambda\Big)(q_I(a_{x,z})\otimes 1)\Big)(q_w\circ q_I)^{\otimes 2}(\Delta(a_{x,y}))\\
		&\qquad= \sum_{1\le k\le d_\lambda}\big(q_w(\nu_x(a_{i,k}^\lambda))\otimes q_w(\nu_z(a_{k,j}^\lambda))\big)(q_w(q_I(a_{x,z}))\otimes 1)(q_w\circ q_I)^{\otimes 2}(\Delta(a_{x,y}))\\
		&\qquad= \sum_{1\le k\le d_\lambda}\big(q_w(\nu_x(a_{i,k}^\lambda))\otimes q_w(\nu_z(a_{k,j}^\lambda))\big)\sum_{z'\in X}q_w(q_I(a_{x,z}a_{x,z'}))\otimes q_w(q_I(a_{z',y})) \\
		&\qquad= \sum_{1\le k\le d_\lambda}\big(q_w(\nu_x(a_{i,k}^\lambda))\otimes q_w(\nu_z(a_{k,j}^\lambda))\big)(q_w(q_I(a_{x,z}))\otimes q_w(q_I(a_{z,y})) )\\
		&\qquad= \sum_{1\le k\le d_\lambda}q_w(\nu_x(a_{i,k}^\lambda)q_I(a_{x,z}))\otimes q_w(\nu_z(a_{k,j}^\lambda)q_I(a_{z,y})).
	\end{align*}
	It follows that
	\begin{align*}
		\Phi_w(a^{(\lambda,X)}_{(i,x),(j,y)})&=\sum_{z\in X}\sum_{1\le k\le d_\lambda}q_w(\nu_x(a_{i,k}^\lambda)q_I(a_{x,z}))\otimes q_w(\nu_z(a_{k,j}^\lambda)q_I(a_{z,y}))\\
		&=\sum_{z\in X}\sum_{1\le k\le d_\lambda}a^{(\lambda,X)}_{(i,x),(k,z)}\otimes a^{(\lambda,X)}_{(k,z),(j,y)},
	\end{align*}
	and so (1) holds for all matrices $a^{(\lambda,X)}$. To see that $a^{(\lambda,X)}$ is invertible, we define $b^{(\lambda,X)}$ by
	\[
	b^{(\lambda,X)}_{(i,x),(j,y)}:=q_w(q_I(a_{y,x})\nu_y((a^\lambda)^{-1}_{i,j})).
	\]
	Then we have 
	\begin{align*}
		(a^{(\lambda,X)}b^{(\lambda,X)})_{(i,x),(j,y)} &= \sum_{z\in X}\sum_{1\le k\le d_\lambda}a^{(\lambda,X)}_{(i,x),(k,z)}b^{(\lambda,X)}_{(k,z),(j,y)}\\
		&= q_w\Big(\sum_{z\in X}\sum_{1\le k\le d_\lambda}\nu_x(a^\lambda_{i,k})q_I(a_{x,z})q_I(a_{y,z})\nu_y((a^\lambda)^{-1}_{k,j})\Big)\\
		&= \delta_{x,y} q_w\Big(\sum_{1\le k\le d_\lambda}\nu_x(a^\lambda_{i,k})\Big(\sum_{z\in X}q_I(a_{x,z})\Big)\nu_x((a^\lambda)^{-1}_{k,j})\Big)\\
		& =\delta_{x,y} q_w\Big(\nu_x\Big(\sum_{1\le k\le d_\lambda}a^\lambda_{i,k}(a^\lambda)^{-1}_{k,j}\Big)\Big)\\
		&=\delta_{x,y}q_w(\nu_x((a^\lambda (a^\lambda)^{-1})_{i,j}))\\
		&= \delta_{x,y}\delta_{i,j}1.
	\end{align*}
	A similar calculation shows that $(b^{(\lambda,X)}a^{(\lambda,X)})_{(i,x),(j,y)} =\delta_{x,y}\delta_{i,j}1$, and so $a^{(\lambda,X)}$ is invertible. Similar calculations also show that $c^{(\lambda,X)}$ with entries
	\[
	c^{(\lambda,X)}_{(i,x),(j,y)}:=q_w(q_I(a_{x,y})\nu_x(((a^\lambda)^T)^{-1}_{i,j}))
	\]
	is the inverse of $(a^{(\lambda,X)})^T$.  
	
	We also have
	\[
	(a^X(a^X)^T)_{x,y}=\sum_{z\in X}a^X_{x,z}a^X_{y,z}=q_w\left(q_I\left(\sum_{z\in X}a_{x,z}a_{y,z}\right)\right)=\delta_{x,y}1.
	\]
	Similarly, $(a^X)^Ta^X$ is the identity. So $a^X$ and $(a^X)^T$ are mutually inverse, and (2) is satisfied. 
	
	We now claim that the entries of the matrices $\{a^{(\lambda,X)}:\lambda\in \Lambda\}\cup\{a^X\}$ span a dense subset of $A \ast_{X,w} \Pp$. For each $x,y\in X$ we obviously have  $q_w(q_I(a_{x,y}))$ in this span since they are the entries of $a^X$. For each $x\in X$, $\lambda\in\Lambda$ and $1\le i,j\le d_\lambda$ we have
	\[
	\sum_{y\in X}a^{(\lambda,X)}_{(i,x),(j,y)}=q_w\left(\nu_x(a^\lambda_{i,j})q_I\left(\sum_{y\in X}a_{x,y}\right)\right)=q_w(\nu_x(a^\lambda_{i,j})),
	\]
	and so each $q_w(\nu_x(a^\lambda_{i,j}))$ is in the span of the entries. The claim follows, and so (3) holds.
\end{proof}

\begin{thm}\label{thm:APWreath}
	Let $A_\Pp$ be a finitely-constrained self-similar quantum group in the sense of Definition~\ref{def:FinitelyConstrainedCQG}. There is a unital quantum group isomorphism $\pi\colon A_\Pp\to A_\Pp\ast_w \Pp$ satisfying 
	\begin{equation}\label{eq:wpeq1}
		\pi(q_J(a_{xu,yv}))=q_w\big(q_I(a_{x,y})\nu_x(q_J(a_{u,v}))\big)
	\end{equation}
	for all $x,y\in X$, $u,v\in X^m$, $m\ge 0$.
\end{thm}

\begin{proof}
	
	We define $b_{\varnothing,\varnothing}$ to be the identity of $A_\Pp\ast_w\Pp$, and for each $x,y\in X$, $u,v\in X^m,m\ge 0$, 
	\[
	b_{xu,yv}:=q_w\big(q_I(a_{x,y})\nu_x(q_J(a_{u,v}))\big).
	\]
	We claim that this gives a family of projections satisfying (1)--(3) of Definition~\ref{def:qaut}. Condition (1) holds by definition. We have 
	\begin{align*}
		b_{xu,yv}^* &= q_w\big(\nu_x(q_J(a_{u,v}^*))q_I(a_{x,y}^*)\big)\\ 
		&= 	q_w\big(\nu_x(q_J(a_{u,v}))q_I(a_{x,y})\big)\\ 
		&= q_w\big(q_I(a_{x,y})\nu_x(q_J(a_{u,v}))\big)\\ 
		&= b_{xu,yv}
	\end{align*}
	and
	\[
	b_{xu,yv}^2=q_w\big(q_I(a_{x,y})\nu_x(q_J(a_{u,v}))q_I(a_{x,y})\nu_x(q_J(a_{u,v}))\big)=q_w\big(q_I(a_{x,y}^2)\nu_x(q_J(a_{u,v}^2))\big)=b_{xu,yv}.
	\]
	So (2) holds. For each $w\in X$ we have
	\begin{align*}
		\sum_{z\in X}b_{xuw,yvz} &= \sum_{z\in X} q_w\big(q_I(a_{x,y})\nu_x(q_J(a_{uw,vz}))\big)\\
		&=q_w\Big(q_I(a_{x,y})\nu_x\Big(q_J\Big(\sum_{z\in X}a_{uw,vz}\Big)\Big)\Big)\\
		&= q_w\big(q_I(a_{x,y})\nu_x(q_J(a_{u,v}))\big)\\
		&=b_{xu,yv},
	\end{align*}
	and
	\begin{align*}
		\sum_{z\in X}b_{xuz,yvw} &= \sum_{z\in X} q_w\big(q_I(a_{x,y})\nu_x(q_J(a_{uz,vw}))\big)\\
		&= q_w\Big(q_I(a_{x,y})\nu_x\Big(q_J\Big(\sum_{z\in X}a_{uz,vw}\Big)\Big)\Big)\\
		&= q_w\big(q_I(a_{x,y})\nu_x(q_J(a_{u,v}))\big)\\
		&=b_{xu,yv},
	\end{align*}
	and hence (3) holds. This proves the claim, and hence the universal property of $\A_X$ now gives a homomorphism $\widetilde{\pi}\colon\A_X\to A_\Pp\ast_w\Pp$ satisfying 
	\[
	\widetilde{\pi}(a_{xu,yv})=q_w\big(q_I(a_{x,y})\nu_x(q_J(a_{u,v}))\big),
	\]
	for all $x,y\in X$, $u,v\in X^m$, $m\ge 0$.
	
	We now claim that $J$ is contained in $\ker\widetilde{\pi}$. To see this, fix $w\in X^n$, with $w=w_1w'$ for $w_1\in X, w'\in X^{n-1}$. We first prove the claim that for each $x_k:=a_{u_1,v_1}\cdots a_{u_k,v_k}$, where $k\ge 1$ and each pair $u_i,v_i\in X^{m_i}$ for some $m_i\ge 0$, we have
	\begin{equation}\label{IdentityforJinKER}
		\widetilde{\pi}(\rho_w(x_k))=\sum_{y\in X} q_w\big(q_I(a_{y,w_1})\nu_y(q_J(\rho_{w'}(x_k)))\big).
	\end{equation}
	Let $k=1$. Then 
	\begin{align*}
		\widetilde{\pi}(\rho_w(a_{u_1,v_1})) &= \sum_{y\in X}\sum_{\alpha\in X^{n-1}}\widetilde{\pi}(a_{y\alpha u_1,wv_1})\\
		&= \sum_{y\in X}\sum_{\alpha\in X^{n-1}}q_w\big(q_I(a_{y,w_1})\nu_y(q_J(a_{\alpha u_1,w'v_1}))\big)\\
		&= \sum_{y\in X}q_w\Big(q_I(a_{y,w_1})v_y\Big(q_J\Big(\sum_{\alpha\in X^{n-1}}a_{\alpha u_1, w' v_1}\Big)\Big)\Big)\\
		&= \sum_{y\in X}q_w\big(q_I(a_{y,w_1})\nu_y(q_J(\rho_{w'}(a_{u_1,v_1})))\big),
	\end{align*}
	and so \eqref{IdentityforJinKER} holds for $k=1$. We now assume true for $x_k$, and prove for $x_{k+1}$. Note that for $y,y'\in X$ we have $q_I(a_{y,w_1})q_I(a_{y',w_1})=\delta_{y,y'}q_I(a_{y,w_1})$, and hence 
	\begin{align*}
		&q_w\big(q_I(a_{y,w_1})\nu_y(q_J(\rho_{w'}(x_k)))q_I(a_{y',w_1})\nu_{y'}(q_J(\rho_{w'}(a_{u_{k+1},v_{k+1}})))\big)\\
		&\qquad\qquad=q_w\big(\nu_y(q_J(\rho_{w'}(x_k)))q_I(a_{y,w_1})q_I(a_{y',w_1})\nu_{y'}(q_J(\rho_{w'}(a_{u_{k+1},v_{k+1}})))\big)\\
		&\qquad\qquad=\delta_{y,y'}q_w\big(\nu_y(q_J(\rho_{w'}(x_k)))q_I(a_{y,w_1})\nu_{y}(q_J(\rho_{w'}(a_{u_{k+1},v_{k+1}})))\big)\\
		&\qquad\qquad=\delta_{y,y'}q_w\big(q_I(a_{y',w_1})\nu_y(q_J(\rho_{w'}(x_k)))\nu_{y}(q_J(\rho_{w'}(a_{u_{k+1},v_{k+1}})))\big)\\
		&\qquad\qquad=\delta_{y,y'}q_w\big(q_I(a_{y',w_1})\nu_y(q_J(\rho_{w'}(x_ka_{u_{k+1},v_{k+1}})))\big).
	\end{align*}
	It follows that
	\[
	\widetilde{\pi}(\rho_w(x_{k+1)}) = \widetilde{\pi}(\rho_w(x_k))\widetilde{\pi}(\rho_w(a_{u_{k+1},v_{k+1}}))= \sum_{y\in X}q_w\big(q_I(a_{y',w_1})\nu_y(q_J(\rho_{w'}(x_ka_{u_{k+1},v_{k+1}})))\big),
	\]
	and it follows that \eqref{IdentityforJinKER} holds for all $k$. Since linear combinations of products of the form $x_k$ is a dense subalgebra of $\A_X$, it follows that 
	\[
	\widetilde{\pi}(\rho_w(a))=\sum_{y\in X} q_w\big(q_I(a_{y,w_1})\nu_y(q_J(\rho_{w'}(a)))\big)
	\]
	for all $a\in \A_X$. Now, if $a\in I$, then $\rho_{w'}(a)\in J=\ker q_J$, and hence the above equations shows that $\widetilde{\pi}(\rho_w(a))=0$. Hence $\rho_w(a)\in \ker\widetilde{\pi}$ for all $w\in X^n$ and $a\in I$, and hence $J\subseteq \ker \widetilde{\pi}$. This means $\widetilde{\pi}$ descends to a homomorphism $\pi\colon A_\Pp\to A_\Pp\ast_w\Pp$ satisfying 
	\[
	\pi(q_J(a_{xu,yv}))=q_w\big(q_I(a_{x,y})\nu_x(q_J(a_{u,v}))\big)
	\]
	for all $x,y\in X$, $u,v\in X^m$, $m\ge 0$.
	
	We now show that $\pi$ is an isomorphism by finding an inverse. For each $x\in X$ consider the homomorphism $q_J\circ\sigma_x\colon \A_X\to A_\Pp$, where $\sigma_x$ is the homomorphism from Remark~\ref{rem:thesigmamap}. Since $\sigma_x=\kappa\circ\rho_x\circ\kappa$, and we know from \cite[Remark~2.10]{Wang95} that $\kappa(J)\subseteq J$, it follows that $q_J\circ\sigma_x$ descends to a homomorphism $\phi_x\colon A_\Pp\to  A_\Pp$ satisfying 
	\[
	\phi_x(q_J(a_{u,v}))=q_J(\sigma_x(a_{u,v}))=\sum_{y\in X}q_J(a_{xu,yv}),
	\]
	for all $u,v\in X^m, m\ge 0$. 
	
	Each $\phi_x$, and the map $q_I(a)\mapsto q_J(a)$ from $\Pp$ to $A_\Pp$, now allow us to apply the universal property of the free product $(\ast_{x\in X}A_\Pp)\ast \Pp$ to get a homomorphism $\widetilde{\phi}\colon (\ast_{x\in X}A_\Pp)\ast \Pp\to A_\Pp$ satisfying $\widetilde{\phi}\circ \nu_x=\phi_x$ for each $x\in X$, and $\widetilde{\phi}(q_I(a))=q_J(a)$ for all $a\in \A_1\subseteq \A_X$. We claim that 
	\[
	\widetilde{\phi}\big(\nu_x(q_J(a_{u,v}))q_I(a_{x,y})-q_I(a_{x,y})\nu_x(q_J(a_{u,v})\big)=0,
	\]
	for each $x\in X$, $u,v\in X^m$, $m\in\N$. We have
	\begin{align*}
		&\widetilde{\phi}\big(\nu_x(q_J(a_{u,v}))q_I(a_{x,y})-q_I(a_{x,y})\nu_x(q_J(a_{u,v})\big)\\
		&\hspace{2cm}=\phi_x(q_J(a_{u,v}))q_J(a_{x,y})-q_G(a_{x,y})\phi_x(q_J(a_{u,v}))\\
		&\hspace{2cm}= \sum_{y\in X}q_J(a_{xu,yv})q_J(a_{x,y})-\sum_{y'\in X}q_J(a_{x,y})q_J(a_{xu,y'v})\\
		&\hspace{2cm}= q_J(a_{xu,yv})-q_J(a_{xu,yv})\\
		&\hspace{2cm}=0.
	\end{align*}
	It follows that $\widetilde{\phi}$ descends to a homomorphism $\phi\colon A_\Pp \ast_w \Pp\to A_\Pp$ satisfying 
	\[
	\phi(q_w(\nu_x(q_J(a_{u,v}))))=q_J(\sigma_x(a_{u,v}))=\sum_{y\in X}q_J(a_{xu,yv})
	\]
	for all $x\in X$, $u,v\in X^m,m\ge 0,$ and 
	\[
	\phi(q_w(q_I(a_{x,y})))=q_J(a_{x,y})
	\]
	for all $x,y\in X$. 
	
	We claim that $\pi$ and $\phi$ are mutually inverse. For $x,y\in X$, $u,v\in X^m,m\ge 0,$ we have
	\[
	\phi(\pi(q_J(a_{xu,yv}))) = \phi\big(q_w\big(q_I(a_{x,y})\nu_x(q_J(a_{u,v}))\big)\big)=q_J(a_{x,y})\sum_{y\in X}q_J(a_{xu,yv})=q_J(a_{xu,yv}),
	\]
	and it follows that $\phi\circ\pi$ is the identity on $A_\Pp$. For $x\in X$, $u,v\in X^m,m\ge 0,$ we have
	\begin{align*}
		\pi(\phi(q_w(\nu_x(q_J(a_{u,v})))))&=\pi\Big(\sum_{y\in X}q_J(a_{xu,yv})\Big)\\
		&=\sum_{y\in X}q_w\big(q_I(a_{x,y})\nu_x(q_J(a_{u,v}))\big)\\
		&= q_w\Big(q_I\Big(\sum_{y\in X}a_{x,y}\Big)\nu_x(q_J(a_{u,v}))\Big)\\
		&= q_w(\nu_x(q_J(a_{u,v}))),
	\end{align*}
	and for all $x,w\in X$ we have
	\[
	\pi(\phi(q_w(q_I(a_{x,y}))))=\pi(q_J(a_{x,y}))=q_w(q_I(a_{x,y})).
	\]
	Hence $\pi\circ\phi$ is the identity on $A_\Pp \ast_w \Pp$, and so $\pi$ is an isomorphism.
	
	We now need to show that $\pi$ is a homomorphism of compact quantum groups, which means that $\Delta_w\circ \pi=(\pi\otimes\pi)\circ\Delta_J$, where $\Delta_J$ is the comultiplication on $A_\Pp$. For $x,y\in X$, $u,v\in X^m,m\ge 0$, we have  
	\begin{align*}
		&(\pi\otimes\pi)\circ\Delta_J(q_J(a_{xu,yv}))\\
		&\qquad=\sum_{\alpha^{m+1}}\pi(q_J(a_{xu,\alpha}))\otimes \pi(q_J(a_{\alpha,yv}))\\
		&\qquad=\sum_{z\in X}\sum_{\beta\in X^{m}}\pi(q_J(a_{xu,z\beta}))\otimes \pi(q_J(a_{z\beta,yv}))\\
		&\qquad=\sum_{z\in X}\sum_{\beta\in X^{m}} q_w\big(q_I(a_{x,z})\nu_x(q_J(a_{u,\beta}))\big)\otimes q_w\big(q_I(a_{z,y})\nu_x(q_J(a_{\beta,v}))\big).
	\end{align*}
	We have 
	\[
	\Delta_w\circ\pi(q_J(a_{xu,yv}))=\Delta_w(q_w(q_I(a_{x,y})))\Delta_w(q_w(\nu_x(q_J(a_{u,v}))),
	\]
	where
	\begin{equation}\label{eq:1forCQGiso}
		\Delta_w(q_w(q_I(a_{x,y})))=\sum_{z\in X} q_w(q_I(a_{x,z}))\otimes q_w(q_{\Pp}(a_{z,y}) ),
	\end{equation}
	and
	\begin{align}
		\nonumber \Delta_w(q_w(\nu_x(q_J(a_{u,v})))
		&\nonumber =\sum_{z'\in X} (q_w\otimes q_w)\big((\nu_x \otimes \nu_{z'})(\Delta_J(q_J(a_{u,v})))(q_I(a_{x,z'}) \otimes 1)\big)\\ 
		&\label{eq:2forCQGiso} = \sum_{z'\in X}\sum_{\beta\in X^m}q_w\big(\nu_x(q_J(a_{u,\beta}))q_I(a_{x,z'})\big)\otimes q_w\big(\nu_{z'}(q_J(a_{\beta,v}))\big).
	\end{align}
	A typical summand in the product of  the expressions in \eqref{eq:1forCQGiso} and \eqref{eq:2forCQGiso} is
	\begin{align*}
		&q_w\big(q_I(a_{x,z})\nu_x(q_J(a_{u,\beta}))q_I(a_{x,z'})\big)\otimes q_w\big(q_{\Pp}(a_{z,y}) \nu_{z'}(q_J(a_{\beta,v}))\big)\\
		&\qquad\qquad\qquad = q_w\big(\nu_x(q_J(a_{u,\beta}))q_I(a_{x,z})q_I(a_{x,z'})\big)\otimes q_w\big(q_{\Pp}(a_{z,y}) \nu_{z'}(q_J(a_{\beta,v}))\big)\\
		&\qquad\qquad\qquad =\delta_{z,z'}q_w\big(\nu_x(q_J(a_{u,\beta}))q_I(a_{x,z})\big)\otimes q_w\big(q_{\Pp}(a_{z,y}) \nu_{z}(q_J(a_{\beta,v}))\big)\\
		&\qquad\qquad\qquad =\delta_{z,z'}q_w\big(q_I(a_{x,z})\nu_x(q_J(a_{u,\beta}))\big)\otimes q_w\big(q_{\Pp}(a_{z,y}) \nu_{z}(q_J(a_{\beta,v}))\big).
	\end{align*}
	Hence 
	\begin{align*}
		\Delta_w\circ\pi(q_J(a_{xu,yv}))
		&= \sum_{z\in X} q_w(q_I(a_{x,z}))\otimes q_w(q_{\Pp}(a_{z,y}) )\\
		&= \sum_{z\in X}\sum_{\beta\in X^{m}} q_w\big(q_I(a_{x,z})\nu_x(q_J(a_{u,\beta}))\big)\otimes q_w\big(q_I(a_{z,y})\nu_x(q_J(a_{\beta,v}))\big)\\
		&= (\pi\otimes\pi)\circ\Delta_J(q_J(a_{xu,yv})),
	\end{align*}
	and it follows that $\Delta_w\circ \pi=(\pi\otimes\pi)\circ\Delta_J$.
\end{proof}

\begin{example}\label{eg:examplesfrom4points}
	An immediate consequence of Theorem~\ref{thm:APWreath} is that $A_\Pp$ is noncommutative whenever $\Pp$ is a noncommutative quantum subgroup of $\A_1$. A class of such examples comes from Banica and Bichon's \cite[Theorem~1.1]{BB09}, in which they classify all the quantum subgroups $\Pp$ of $\A_1$ for $|X|=4$; the corresponding list of quantum groups $A_\Pp$ gives us a list of potentially interesting self-similar quantum groups for further study.
\end{example}

\end{document}